\documentclass[11pt]{amsart}
\usepackage{amsrefs}
\usepackage{amssymb}
\usepackage[all,arc]{xy}
\usepackage{graphicx}
\usepackage{mathrsfs}
\usepackage{xcolor}

\usepackage{multicol}

\usepackage{mathpazo}
\usepackage{a4wide}

\DeclareMathOperator{\cd}{cd}

\DeclareMathOperator{\Hom}{Hom}

\DeclareMathOperator{\image}{Im}

\newcommand{\euV}{\mathscr{V}}
\newcommand{\euE}{\mathscr{E}}

\newcommand{\K}{\mathbb{K}}

\newcommand{\blue}[1]{{#1}}

\DeclareFontFamily{U}{wncy}{}
\DeclareFontShape{U}{wncy}{m}{n}{<->wncyr10}{}
\DeclareSymbolFont{mcy}{U}{wncy}{m}{n}
\DeclareMathSymbol{\Sha}{\mathord}{mcy}{"58}
\DeclareMathSymbol{\sha}{\mathord}{mcy}{"78}

\begin{document}

\newtheorem{thm}{Theorem}[section]
\newtheorem{cor}[thm]{Corollary}
\newtheorem{lem}[thm]{Lemma}
\newtheorem{prop}[thm]{Proposition}
\newtheorem{defin}[thm]{Definition}
\newtheorem{exam}[thm]{Example}
\newtheorem{examples}[thm]{Examples}
\newtheorem{rem}[thm]{Remark}
\newtheorem{case}{\sl Case}
\newtheorem{claim}{Claim}
\newtheorem{fact}[thm]{Fact}
\newtheorem{question}[thm]{Question}
\newtheorem{conj}[thm]{Conjecture}
\newtheorem*{notation}{Notation}
\swapnumbers
\newtheorem{rems}[thm]{Remarks}

\newtheorem{questions}[thm]{Questions}
\numberwithin{equation}{section}

\newcommand{\gr}{\mathrm{gr}}
\newcommand{\isom}{\cong}
\newcommand{\dbC}{\mathbb{C}}
\newcommand{\F}{\mathbb{F}}
\newcommand{\dbN}{\mathbb{N}}
\newcommand{\Q}{\mathbb{Q}}
\newcommand{\dbR}{\mathbb{R}}
\newcommand{\dbU}{\mathbb{U}}
\newcommand{\Z}{\mathbb{Z}}
\newcommand{\dbl}{[\![}
\newcommand{\dbr}{]\!]}

\newcommand{\Lie}{\mathfrak{L}}
\newcommand{\rlie}{\mathfrak{r}}
\newcommand{\calX}{\mathcal{X}}
\newcommand{\calU}{\mathcal{U}}
\newcommand{\FpX}{\mathbb{F}_p\langle X\rangle}
\newcommand{\FppG}{\mathbb{F}_p[\![G]\!]}
\newcommand{\FppX}{\mathbb{F}_p\langle\!\langle X\rangle\!\rangle}
\newcommand{\bfT}{\mathbf{T}}
\newcommand{\bfLambda}{\mathbf{\Lambda}}
\newcommand{\bfA}{\mathbf{A}}
\newcommand{\bfS}{\mathbf{S}}
\newcommand{\bfQ}{\mathbf{Q}}


\title[Mild pro-$p$ groups and Koszulity]{Mild pro-$p$ groups  and the Koszulity conjectures}
\author{J. Min\'a\v{c}}
\author{F. W. Pasini}
\author{C. Quadrelli}
\author{N. D. T\^an}
\address{Department of Mathematics, Western University, London ON, Canada}
\email{minac@uwo.ca}
\address{Brain and Mind Institute, Western University, London ON, Canada}
\email{fpasini@uwo.ca}
\address{Department of of Science and High-Tech, University of Insubria, Como, Italy-EU}
\email{claudio.quadrelli@uninsubria.it}
\address{School of Applied Mathematics and Informatics, Hanoi University of Science and Technology, Hanoi, Vietnam} 
\email{tan.nguyenduy@hust.edu.vn}
\date{\today}

\begin{abstract}
 Let $p$ be a prime, and $\F_p$ the field with $p$ elements.
We prove that if $G$ is a mild pro-$p$ group with quadratic $\F_p$-cohomology algebra
$H^\bullet(G,\F_p)$, then the algebras $H^\bullet(G,\F_p)$ and $\gr\FppG$ --- the latter being induced by the quotients of consecutive terms
of the $p$-Zassenhaus filtration of $G$ ---
are both Koszul, and they are quadratically dual to each other. 
Consequently, if the maximal pro-$p$ Galois group of a field is mild, then Positselski's and Weigel's Koszulity conjectures
hold true for such a field.
\end{abstract}

\dedicatory{To the memory of Professor Mikhail Shubin \\ with gratitude and admiration}



\thanks{The first-named author is partially supported by the Natural Sciences and Engineering Research Council of Canada, grant R0370A01. The fourth-named author is partially supported by Vietnam’s National Foundation for Science and Technology Development (NAFOSTED) under grant number 101.04-2019.314.}

\maketitle

\section{Introduction}
\label{sec:intro}

Let $\Bbbk$ be a field, and let $A_\bullet=\bigoplus_{n\geq0}A_n$ be a $\Bbbk$-algebra endowed with a $\mathbb{N}$-grading.
The algebra $A_\bullet$ is called a {\sl quadratic algebra } if it is generated by $A_1$
(namely, every element is a sum of products of elements of $A_1$),
and its relations are generated by homogeneous relations of degree 2.
A quadratic algebra $A_\bullet$ comes endowed with a {\sl quadratic dual algebra} $A_\bullet^!$, defined
as the quadratic algebra whose relations are orthogonal to the relations of $A_\bullet$ (see Definition~\ref{defin:quad} below,
and \cite[\S~1.2]{poliposi:book}). 
For example, symmetric algebras and exterior algebras are quadratic, and they are quadratically dual to each other. 

A peculiar class of algebras stands out among quadratic algebras: the class of {\sl Koszul algebras}, introduced by S.~Priddy in \cite{priddy}.
A quadratic algebra $A_\bullet$ is Koszul if the minimal graded free $A_\bullet$-resolution of $\Bbbk$ has only linear maps --- see \S~\ref{ssec:koszul} below.
In spite of the kinky definition, Koszul algebras have an exceptionally nice behavior --- e.g., the $\Bbbk$-cohomology
algebra of a Koszul algebra $A_\bullet$ is its quadratic dual $A_\bullet^!$.
Koszulity is a very restrictive property, still it arises in various areas of mathematics, such as representation
theory, algebraic geometry, combinatorics, noncommutative geometry and topology (see, e.g., \cite{koszul:survey} and \cite[Introduction]{poliposi:book}).
Hence, Koszul algebras have become an important object of study.

Quadratic algebras and Koszul algebras play a prominent role in Galois theory, too.
Given a field $\K$, for $p$ a prime number let $G_{\K}(p)$ denote the {\sl maximal pro-$p$ Galois group} of $\K$.
I.e., $G_{\K}(p)$ is the maximal pro-$p$ quotient of the absolute Galois group of $\K$.
If $\K$ contains a root of 1 of order $p$ (and also $\sqrt{-1}$ if $p=2$), then the celebrated {\sl Bloch-Kato Conjecture} --- which was given an affirmative answer by M.~Rost and V.~Voevodsky, with a ``patch'' by Ch.~Weibel, cf.~\cites{rost,HW:book,weibel:lectures,weibel:norm,voev} --- implies that the {\sl $\F_p$-cohomology algebra}
$H^\bullet(G_{\K}(p),\F_p)=\bigoplus_{n\geq0}H^n(G_{\K}(p),\F_p)$
of the maximal pro-$p$ Galois group of $\K$, endowed with the graded-commutative {\sl cup product}
\[
 H^s(G_{\K}(p),\F_p)\times H^t(G_{\K}(p),\F_p)\overset{\cup}{\longrightarrow} H^{s+t}(G_{\K}(p),\F_p),\qquad s,t\geq0,
\]
is a quadratic $\F_p$-algebra.
This result was a great achievement in Galois theory, as it allowed the discovery of new insights on the structure
of maximal pro-$p$ Galois groups of fields, whose study is crucial in current research in Galois theory (see, e.g., \cites{CEM,cq:bk,CMQ:fast,qw:cyc}).

Given the peculiarity of Koszul algebras among quadratic algebras, the study of Koszul algebras in the context of Galois cohomology has gained considerable importance, in particular after the works by L.~Positselski and A.~Vishik (cf.~\cites{posivis,pos:k,pos:conj,posi:form}).
Their work propose an approach to the Bloch-Kato Conjecture which is alternative to Rost-Voevodsky's proof of the Norm Residue Theorem (see \cite[\S~1]{pos:k}).
In particular, Positselski formulated the following conjecture (cf.~\cite{pos:conj}*{\S~0.1}).

\begin{conj}\label{conj:posi}
Let $\K$ be a field containing a root of 1 of order $p$.
Then the algebra $H^\bullet(G_{\K}(p),\F_p)$ is Koszul. 
\end{conj}

Since the $\F_p$-cohomology algebra of the maximal pro-$p$ Galois group $G_{\K}(p)$ of a field $\K$ containing a root of 1 of order $p$ is quadratic, a further natural question one may formulate is: ``what is the quadratic dual of $H^\bullet(G_{\K}(p),\F_p)$ like?''
The conjectural description of such quadratic algebra is related to the successive factors of the {\sl $p$-Zassenhaus filtration} of $G_{\K}(p)$ (see Definition~\ref{defin:zassenhaus} below).
For a pro-$p$ group $G$, let $\calU_{\Lie(G)}$ denote the restricted universal enveloping algebra of the restricted
Lie algebra $\Lie(G)$ induced by the $p$-Zassenhaus filtration of $G$. 
One has the following conjectures, formulated respectively by the authors in \cite{MPQT}, and by Th.~Weigel in \cite{thomas:koszul}.

\begin{conj}\label{ques:koszul}
 Let $\K$ be a field containing a root of 1 of order $p$.
\begin{itemize}
 \item[(i)] The algebra $\calU_{\Lie(G_{\K}(p))}$ is quadratic, and it is isomorphic to the quadratic dual of $H^\bullet(G_{\K}(p),\F_p)$.
 \item[(ii)] The algebra $\calU_{\Lie(G_{\K}(p))}$ is Koszul.
\end{itemize}
\end{conj}

Note that, by a result of S.~Jennings, for a finitely generated pro-$p$ group $G$, the algebra $\calU_{\Lie(G)}$ is 
isomorphic to the graded algebra $\gr\F_p[G]$ associated to the filtration on the $\F_p$-group algebra of $G$
by successive powers of the augmentation ideal (see Proposition~\ref{prop:grFG}).

In \cite{pos:conj} it is shown that Conjecture~\ref{conj:posi} holds true if $\K$ is a local field or a global field.
Furthermore, in \cite{MPQT} and in \cite{cq:onerel}, Conjectures~\ref{conj:posi} and \ref{ques:koszul} receive positive answers
in several relevant cases: for example, in case $G_{\K}(p)$ is a {\sl Demushkin group} (cf.~\cite{MPQT}*{Thm.~C}); or if $G$ is finitely generated and it has a single {\sl defining relation}, and $H^\bullet(G_{\K}(p),\F_p)$ is quadratic (cf.~\cite{cq:onerel}*{Cor.~1.3}).

In this paper, we turn our focus on {\sl mild pro-$p$ groups}.
{ Mild pro-$p$ groups --- defined by J.~Labute in \cite{labute:mild} --- are finitely generated pro-$p$ groups which admit a presentation in terms of generators and defining relations, where the relations satisfy a condition of being free in a maximal possible way with respect to suitably chosen filtrations, such as the $p$-Zassenhaus filtration.
Mild pro-$p$ groups enjoy very nice properties: for example, a mild pro-$p$ group $G$ has cohomological dimension $2$ --- i.e.,  $H^n(G,\F_p)=0$ for $n\geq3$ ---, and the associated algebra $\calU_{\Lie(G)}$ may be computed easily.
For this reasons, from a group-theoretic point of view mild pro-$p$ groups may be seen as a generalization of Demushkin groups
and one-relator pro-$p$ groups.
 }

As underlined in \cite[\S~1.2]{labute:mild}, mild pro-$p$ groups arise as Galois groups
of extensions of number fields with restricted ramification { (see also, e.g., \cite{salle:mild,maire}) --- which are very important objects in algebraic number theory, and before Labute's work very little was known about the structure of these Galois pro-$p$ groups.
Most remarkably, mild pro-$p$ groups provide examples of Galois pro-$p$ groups of number fields with cohomological dimension 2, a property which is very interesting as it yields a lot of arithmetic consequences --- see, e.g., \cite{labute:algebres,koch:2,schmidt:circ,forre}.}

Our main result is the following.

\begin{thm}\label{thm:main intro}
Let $G$ be a mild pro-$p$ group.
If the $\F_p$-cohomology algebra $H^\bullet(G,\F_p)$ is quadratic, then the algebras $H^\bullet(G,\F_p)$
and $\calU_{\Lie(G)}$ are Koszul, and moreover one has an isomorphism of quadratic algebras
\[
 H^\bullet(G,\F_p)^!\cong \calU_{\Lie(G)}.\]
\end{thm}

As a consequence, Conjectures~\ref{conj:posi} and \ref{ques:koszul} have positive answers
in the case the maximal pro-$p$ Galois group $G_{\K}(p)$ of a field $\K$ containing a root of 1 of order $p$ is mild.
Moreover, Theorem~\ref{thm:main intro} provides a big wealth of examples of finitely generated pro-$p$ groups $G$
such that the algebras $H^\bullet(G,\F_p)$ and $\calU_{\Lie(G)}$ are Koszul, and quadratically dual to each other 
(see \S~\ref{sec:examples mild}).
In addition, we prove the following cohomological characterization of mild pro-$p$ groups with quadratic cohomology.

\begin{prop}\label{prop:mild cohom intro}
 Let $G$ be a finitely generated pro-$p$ group of cohomological dimension $2$ such that the $\F_p$-cohomology algebra $H^\bullet(G,\F_p)$ is Koszul.
 Then $G$ is a mild pro-$p$ group.
\end{prop}

This raises the following natural Galois-theoretic question.

\begin{question}
Let $p$ be a prime, and let $\K$ be a field such that:
\begin{itemize}
 \item[(i)] $\K$ contains a root of 1 of order $p$, and also $\sqrt{-1}$ if $p=2$;
 \item[(ii)] the maximal pro-$p$ Galois group $G_{\K}(p)$ is a finitely generated pro-$p$ group;
 \item[(iii)] the cohomological dimension of $G_{\K}(p)$ is 2.
\end{itemize}
Is $G_{\K}(p)$ a mild pro-$p$ group?
\end{question}

{ In order to decide when a given pro-$p$ group is mild, there are some criteria which provide sufficient conditions for mildness: for example, Anick's criterion (see \cite{anick} and \cite[Prop.~3.7]{labute:mild}), Labute's non-singular circuits criterion (see \cite[Thm.~1.6]{labute:mild}), and Schmidt's cup-product criterion (see \cite[Thm.~1.1]{Schmidt} and \S~\ref{ssec:cupproductcrit}).
Similarly, there are some methods to prove Koszulity of a quadratic algebra (see, e.g., \cite[Ch.~4]{lodval:book}).
We establish the following new criterion to decide when a given quadratic algebra is Koszul, which is inspired by the aforementioned cup-product criterion for mildness.

\begin{thm}\label{thm:Koszulcrit}
  Let $A_\bullet$ be a quadratic algebra over an arbitrary field $\Bbbk$, and suppose that the component $A_1$ admits a decomposition as direct sum of vector spaces $A_1=U\oplus W$ such that:
 \begin{itemize}
  \item[(i)] $U\cdot W:=\mathrm{Span}_{\F}\{\:u\cdot w\:\mid\:u\in U,\:w\in W\:\}=A_2$;
  \item[(ii)] $U\cdot U:=\mathrm{Span}_{\F}\{\:u\cdot u'\:\mid\:u,u'\in U\:\}=0$.
 \end{itemize}
Then the algebra $A_\bullet$ is Koszul.
\end{thm}

}


\subsection*{Acknowledgements}
{\small The authors are deeply indebted with John~P.~Labute, as their way of thinking about pro-$p$ groups with Koszul $\F_p$-cohomology has been toroughly influenced by his works.
Also, they wish to thank Thomas~S.~Weigel, for the interesting discussions on Galois cohomology.

\noindent
The first-named author gratefully acknowledges the Faculty of Science Distinguished Research Professorship, Western University, for the academic year 2020/2021.
The fourth-named author gratefully acknowledges the Vietnam Institute for Advanced Study in Mathematics (VIASM) for hospitality and support during a visit in 2021.

\noindent
 Last, but not least, we acknowledge gratefully encouraging referee's comments and his/her valuable suggestions concerning the exposition.
}


\section{Quadratic algebras and Koszul algebras}
\label{sec:quad}

Throughout the paper every graded algebra 
$A_\bullet=\bigoplus_{n\in\Z}A_n$ is assumed to be
unitary associative over {a field $\Bbbk$}, non-negatively graded, connected, and of finite-type.
I.e., $A_0=\Bbbk$, $A_n=0$ for $n<0$ and $\dim(A_n)<\infty$ for $n\geq1$.
{From \S~\ref{sec:gps} onward, we will assume $\Bbbk=\F_p$.}

For a complete account on graded algebras and their cohomology, we direct the reader to \cites{poliposi:book,lodval:book,koszul:survey}, and to \cite{MPQT}*{\S~2}.


\subsection{Quadratic algebras and quadratic duals}
\label{ssec:quad}

For a vector space $V$ of finite dimension, let $\bfT_\bullet(V)=\bigoplus_{n\geq0}V^{\otimes n}$ denote
the graded tensor algebra generated by $V$ (where $V^{\otimes0}=\Bbbk$), endowed with the multiplication and the grading
induced by the tensor product.
Also, let $V^*=\Hom_{\F_p}(V,\Bbbk)$ denote the dual space of $V$.
Since $\dim (V)<\infty$, we will identify $(V\otimes V)^*=V^*\otimes V^*$.

\begin{defin}\label{defin:quad}\rm
\begin{itemize}
 \item[(i)] A graded algebra $A_\bullet$ is said to be {\sl quadratic} if 
\[
 A_\bullet\cong \bfT_\bullet(A_1)/(\Omega)
\]
for some { vector subspace} $\Omega\subseteq V\otimes V$, with $(\Omega)\trianglelefteq \bfT_\bullet(A_1)$
the two-sided ideal generated by $\Omega$.
We write $A_\bullet=\bfQ_\bullet(V,\Omega)$, with $V=A_1$. 
\item[(ii)] For a quadratic algebra $A_\bullet=\bfQ_\bullet(V,\Omega)$, let $\Omega^\perp\subseteq V^*\otimes V^*$ denote the orthogonal space
of $\Omega$, i.e., {identifying $(V\otimes V)^*=V^*\otimes V^*$,}
$$\Omega^\perp=\left\{\:\alpha\in(V\otimes V)^*\:\mid\: \alpha(w)=0\text{ for all }w\in \Omega\:\right\}.$$
The {\sl quadratic dual} of $A_\bullet$, denoted by $A_\bullet^!$, is the algebra
$\bfQ_\bullet(V^*,\Omega^\perp)$.
\end{itemize}
\end{defin}

\begin{rem}\label{rem:dual}\rm
For a quadratic algebra $A_\bullet$ one has $(A_\bullet^!)^!=A_\bullet$ (cf.~\cite{lodval:book}*{\S~3.2.4}). 
\end{rem}

Let $\Bbbk\langle X\rangle$ denote the algebra of non-commutative polynomials {with coefficients in $\Bbbk$} on the set of indeterminates $X=\{X_1,\ldots,X_d\}$.
Then $\Bbbk\langle X\rangle$ is a graded algebra, with the grading induced by polynomial degree.
The identification of $X$ with a fixed basis of a vector space $V$ of dimension $d$ yields an isomorphism
of graded algebras $\Bbbk\langle X\rangle\cong \bfT_\bullet(V)$.
Thus, one may identify a quadratic algebra with the quotient of $\Bbbk\langle X\rangle$ by an ideal generated by homogeneous {(non-commutative)} polynomials of degree 2.

\begin{exam}\label{ex:quadratic}\rm
Let $V$ be a finite-dimensional vector space.
\begin{itemize}
 \item[(a)] The tensor algebra $\bfT_\bullet(V)=\bfT_\bullet(V)/(0)$ and the {\sl trivial quadratic algebra} $\bfQ_\bullet(V,V^{\otimes2})=\bfT_\bullet(V)/(V^{\otimes2})$
are quadratic, and $\bfT_\bullet(V)^!\cong \bfQ_\bullet(V,V^{\otimes2})$.
\item[(b)] The symmetric algebra $\bfS_\bullet(V)$ and the exterior algebra $\bfLambda_\bullet(V)$ are quadratic, as
\[ \begin{split}
     \bfS_\bullet(V)&=\frac{\bfT_\bullet(V)}{(\:v\otimes w-w\otimes v\:\mid\:v,w\in V\:)},\\
\bfLambda_\bullet(V)&=\frac{\bfT_\bullet(V)}{(\:v\otimes v\:\mid\:v\in V\:)};
   \end{split}\]
 moreover, $\bfS_\bullet(V)^!\cong\bfLambda_\bullet(V)$.
 \end{itemize}
 \end{exam}
 
One may combine two quadratic algebras $A_\bullet=\bfQ_\bullet(A_1,\Omega_A)$ and $B_\bullet=\bfQ_\bullet(B_1,\Omega_B)$ as follows:
\begin{itemize}
 \item[(a)]  the {\sl direct sum} of $A_\bullet$ and $B_\bullet$ is the quadratic algebra
\[
 A_\bullet\oplus B_\bullet=\bfQ_\bullet(A_1\oplus B_1,\Omega_A\cup\Omega_B\cup A_1\otimes B_1\cup B_1\otimes A_1);
\]
 \item[(b)]  the {\sl free product} of $A_\bullet$ and $B_\bullet$ is the quadratic algebra
\[
 A_\bullet\sqcup B_\bullet=\bfQ_\bullet(A_1\oplus B_1,\Omega_A\cup\Omega_B);
\]
 \item[(c)]  the {\sl symmetric tensor product} of $A_\bullet$ and $B_\bullet$ is the quadratic algebra
\[
 A_\bullet\otimes B_\bullet=\bfQ_\bullet(A_1\oplus B_1,\Omega_A\cup\Omega_B\cup \Omega_{\mathrm{S}}),
\]
where $\Omega_{\mathrm{S}}=\langle a\otimes b-b\otimes a,a\in A_1,b\in B_1\rangle$.
 \item[(d)]  the {\sl wedge product} of $A_\bullet$ and $B_\bullet$ is the quadratic algebra
\[
 A_\bullet\wedge B_\bullet=\bfQ_\bullet(A_1\oplus B_1,\Omega_A\cup\Omega_B\cup\Omega_\Lambda),
\]
where $\Omega_{\Lambda}=\langle a\otimes b+b\otimes a,a\in A_1,b\in B_1\rangle$.
\end{itemize}
One has the isomorphisms of quadratic algebras
\begin{equation}\label{eq:quaddual operations}
\begin{split}
&A_\bullet^!\oplus B_\bullet^!\cong(A_\bullet\sqcup B_\bullet)^! \qquad\text{and}\qquad
A_\bullet^!\sqcup B_\bullet^!\cong(A_\bullet\oplus B_\bullet)^!,\\
& A_\bullet^!\otimes B_\bullet^!\cong(A_\bullet\wedge B_\bullet)^! \qquad\text{and}\qquad
A_\bullet^!\wedge B_\bullet^!\cong(A_\bullet\otimes B_\bullet)^!
\end{split}
\end{equation}
(cf.~\cite{poliposi:book}*{\S~3.1}).

\begin{exam}\label{ex:RAAGs}\rm
Let $\Gamma=(\euV,\euE)$ be a finite {simplicial} graph, with vertices 
$\euV=\{v_1,\ldots,v_d\}$ and edges $\euE$ (for the notion of graph we refer to \cite{diestel}*{\S~1.1}), and let $V$ be the vector space {(over an arbitrary field)} with basis $\euV$.
The {\sl right-angled Artin algebra} $\bfA_\bullet(\Gamma)$ and the {\sl exterior Stanley-Reisner algebra} $\bfLambda_\bullet(\Gamma)$
associated to $\Gamma$ are the quadratic algebras
\[ \begin{split}
  \bfA_\bullet(\Gamma) &=\frac{\bfT_\bullet(V)}{(\:v\otimes w-w\otimes v\:\mid\: \{v,w\}\in \euE\:)},\\
 \bfLambda_\bullet(\Gamma) &= \frac{\bfLambda_\bullet(V)}{(\:v\otimes w\:\mid\: \{v,w\}\notin \euE\:)}.
 \end{split}\]
Then $\bfA_\bullet(\Gamma)^!\cong \bfLambda_\bullet(\Gamma)$ (cf.~\cite{thomas:graded}*{\S~4.2.2}).
In particular,
\begin{itemize}
 \item[(a)] if $\Gamma$ is complete (i.e., $\euE=\{\{v,w\}\mid v,w\in\euV{,v\neq w}\}$) then $\bfA_\bullet(\Gamma)\cong \bfS_\bullet(V)$ and
$\bfLambda_\bullet(\Gamma)\cong\bfLambda_\bullet(V)$;
 \item[(b)] if $\euE=\varnothing$, then $\bfA_\bullet(\Gamma)\cong \bfT_\bullet(V)$
and $\bfLambda_\bullet(\Gamma)\cong \bfQ_\bullet(V,V^{\otimes2})$.
\end{itemize}
\end{exam}


\subsection{Koszul algebras}
\label{ssec:koszul}
A quadratic algebra $A_\bullet$ is said to be {\sl Koszul} if {the trivial module $\Bbbk$ admits a free resolution}
\[
 \xymatrix{ \cdots\ar[r] & P(2)_\bullet\ar[r] & P(1)_\bullet\ar[r] & P(0)_\bullet \ar[r] & \F_p }
\]
of {graded} right $A_\bullet$-modules {$P(i)_\bullet=\bigoplus_{n\geq0}P(i)_n$, such that, for each $n\in\dbN$, $P(n)_n$ is finite-dimensional over $\Bbbk$ and generates $P(n)_\bullet$ as graded $A_\bullet$-module (cf. \cite{MPQT}*{\S~2.2}).}

Koszul algebras have an exceptionally nice behavior in terms of cohomology: indeed, a quadratic algebra $A_\bullet$ is Koszul if, and only if, the Ext-groups $\mathrm{Ext}_{A_\bullet}^{i,j}(\Bbbk,\Bbbk)$ are trivial for $i\neq j$ (cf.~\cite{poliposi:book}*{Ch.~2, Def.~1}).
Moreover, Koszul algebras have the following fundamental property (cf.~\cite{poliposi:book}*{Ch.~2, \blue{Def.~1 and }Cor.~3.3}), which is essentially due to S. Priddy~\cite{priddy}*{\S 2}. (See also 
\cite{BGS}*{\S 2} for a very nice exposition of some classical results.)

\begin{prop}\label{prop:koszul}
A quadratic algebra $A_\bullet$ is Koszul if, and only if, the quadratic dual $A_\bullet^!$ is Koszul.
Moreover, if $A_\bullet$ is Koszul, then the isomorphism of vector spaces $\mathrm{Ext}_{A_\bullet}^{1,1}(\Bbbk,\Bbbk)\cong A_1^*$ yields an isomorphism of quadratic algebras
$$ \bigoplus_{i\geq0}\mathrm{Ext}_{A_\bullet}^{i,i}(\Bbbk,\Bbbk)\cong A_\bullet^!.$$
\end{prop}

\begin{exam}\label{ex:koszul}\rm
Let $V$ be a finite-dimensional vector space.
\begin{itemize}
 \item[(a)] The tensor algebra $\bfT_\bullet(V)$ and the algebra $\bfQ_\bullet(V,V^{\otimes2})$
are Koszul (cf.~\cite{MPQT}*{Ex.~2.9}).
\item[(b)] The symmetric algebra $\bfS_\bullet(V)$ and the exterior algebra $\bfLambda_\bullet(V)$ are Koszul (cf.~\cite{MPQT}*{Ex.~2.9}.
\end{itemize}
\end{exam}

\begin{exam}\label{ex:RAAG koszul}\rm
 Let $\Gamma=(\euV,\euE)$ be a finite simplicial graph.
 Then by the main result in \cite{froberg} the algebras $\bfA_\bullet(\Gamma)$ and $\bfLambda_\bullet(\Gamma)$ {of Example \ref{ex:RAAGs}} are Koszul (see also \cites{papa,thomas:graded}).
\end{exam}

\begin{exam}\label{ex:demushkin algebra}\rm
{ Put $\Bbbk=\F_p$.} 
For $d\geq1$ set $X=\{X_1,\ldots,X_d\}$ and $[X_i,X_j]=X_iX_j-X_jX_i$.
\begin{itemize}
 \item[(a)] If $d$ is even, let $f$ be the homogeneous polynomial of degree 2
 \[
  f=[X_1,X_2]+[X_3,X_4]+\ldots+[X_{d-1},X_d].
 \]
 \item[(b)] If $p=2$ and $d$ is even, let $f$ be the homogeneous polynomial of degree 2
 \[
  f=X_1^2+[X_1,X_2]+[X_3,X_4]+\ldots+[X_{d-1},X_d].
 \]
\item[(c)] If $p=2$ and $d$ is odd, let $f$ be the homogeneous polynomial of degree 2
 \[
  f=X_1^2+[X_2,X_3]+[X_4,X_5]+\ldots+[X_{d-1},X_d].
 \]
 \end{itemize} 
 Then the algebra $\FpX/(f)$ is Koszul.
These Koszul algebras are called {\sl Demushkin graded $\F_p$-algebras} (cf.~\cite{MPQT}*{\S~5.2} and \cite{cq:onerel}*{\S~4.1}).
\end{exam}

{
\begin{exam}\label{ex:demushkin algebra bis}\rm
Put $\Bbbk=\F_p$.
For $d\geq1$ set $X=\{X_1,\ldots,X_d\}$ and $[X_i,X_j]=X_iX_j-X_jX_i$. In each of the cases
\begin{itemize}
 \item[(a)] if $d$ is even and $f$ is the homogeneous polynomial of degree $2$
 \[
  f=[X_1,X_2]+[X_3,X_4]+\ldots+[X_{d-1},X_d]
 \]
 \item[(b)] if $p=2$, $d$ is even, and $f$ is the homogeneous polynomial of degree $2$
 \[
  f=X_1^2+[X_1,X_2]+[X_3,X_4]+\ldots+[X_{d-1},X_d]
 \]
\item[(c)] if $p=2$, $d$ is odd, and $f$ is the homogeneous polynomial of degree $2$
 \[
  f=X_1^2+[X_2,X_3]+[X_4,X_5]+\ldots+[X_{d-1},X_d]
 \]
 the {\sl one-relator algebra} $\FpX/(f)$ is Koszul.
 \end{itemize} 

These Koszul algebras are called {\sl Demushkin graded $\F_p$-algebras} (cf.~\cite{MPQT}*{\S~5.2} and \cite{cq:onerel}*{\S~4.1}).
\end{exam}
}

Moreover, Koszulity is preserved by taking direct sums and free products (cf.~\cite{poliposi:book}*{Ch.~3, Cor.~1.2}).

\begin{prop}\label{prop:koszul operations}
 If $A_\bullet$ and $B_\bullet$ are two Koszul algebras, then also the direct sum $A_\bullet\oplus B_\bullet$ and the free product $A_\bullet\sqcup B_\bullet$ are Koszul.
\end{prop}

For a graded algebra $A_\bullet$, the {\sl Hilbert series} $h_{A_\bullet}(z)$ is the formal power series
\[
 h_{A_\bullet}(z)=\sum_{n\geq0}\dim(A_n)z^n.
\]
\blue{ There is a rich interplay between Hilbert series and Koszul property: the following two facts are of key importance for the main results of this paper} (cf.~\cite{poliposi:book}*{Ch.~2, \blue{Cor.~2.2, and Prop.~2.3}}).

\begin{prop}\label{prop:koszul A3}
 Let $A_\bullet=\blue{\bfQ_\bullet(V,\Omega)}$ be a quadratic algebra \blue{ with $d=\dim(V)$ generators and $r=\dim(\Omega)$ defining relations. 
\begin{itemize} 
\item[(i)]  If $h_{A_\bullet}(z)=(1-dz+rz^2)^{-1}$, then $A_\bullet$ is Koszul.
\item[(ii)] If $A_\bullet$ is Koszul, then one has an equality of formal power series
\[
h_{A_\bullet}(z)\cdot h_{A_\bullet^!}(-z)=1.
\]
\end{itemize}
}
\end{prop}


\section{Pro-$p$ groups and graded algebras}\label{sec:gps}

Henceforth, we focus on finitely generated pro-$p$ groups.
Thus, every subgroup is tacitly assumed to be closed, and generators are to be considered in the topological sense.
In particular, for a pro-$p$ group $G$ and $n\geq1$, $G^n$ denotes the closed subgroup of $G$ generated by $n$th powers,
and, for $H_1,H_2$ two subgroups of $G$, $[H_1,H_2]$ denotes the closed subgroup of $G$ generated by the commutators
\[
 g_1^{-1}\cdot g_1^{g_2}=g_1^{-1}g_2^{-1}g_1g_2,\qquad g_1\in H_1,g_2\in H_2.
\]

\subsection{The Zassenhaus filtration}
\label{ssec:zass}

\begin{defin}\label{defin:zassenhaus}\rm
The {\sl $p$-Zassenhaus filtration} $(G_{(n)})_{n\geq1}$ of a pro-$p$ group $G$ is 
defined by
\begin{equation}\label{eq:zassenhaus defi}
\begin{split}
G_{(1)} &= G \\
 G_{(n)}&= G_{(\lceil n/p\rceil)}^p\cdot\prod_{h+k=n}[G_{(h)},G_{(k)}], \qquad n\geq2,
\end{split}
\end{equation}
where $\lceil n/p\rceil$ is the least integer $l$ such that $lp\geq n$.
\end{defin}

The $p$-Zassenhaus filtration of $G$ is the fastest descending series starting at $G$ satisfying
\begin{equation}\label{eq:prop zassenhaus}
[G_{(h)},G_{(k)}]\subseteq G_{(h+k)}\quad\text{and}\quad G_{(n)}^p\subseteq G_{(np)}
\end{equation}
for every $h,k,n\geq 1$.
Moreover, every quotient $G_{(n)}/G_{(n+1)}$ is a $p$-elementary abelian group, and thus a vector space over $\F_p$.
In particular, one has $G_{(2)}=G^p[G,G]$, and 
\[
 G_{(3)}=\begin{cases} G^p[G,[G,G]] & \text{if }p\neq2 \\ G^4[G,G]^2[G,[G,G]] & \text{if }p=2.\end{cases} 
\]
We set $\Lie(G)=\bigoplus_{n\geq1}G_{(n)}/G_{(n+1)}$.

Let $S$ be a free pro-$p$ group with basis $\calX=\{x_1,\ldots,x_d\}$.
Then one may define the $p$-Zassenhaus filtration of $S$ as follows.
Let $\F_p\dbl S\dbr$ denote the complete group algebra of $S$ --- namely, $\F_p\dbl S\dbr=\varprojlim_N\F_p[S/N]$,
where $N$ runs through the open normal subgroups of $S$.
The assignment $x\mapsto x-1$ induces an embedding of sets $S\hookrightarrow \F_p\dbl S\dbr$.

Given a set of indeterminates $X=\{X_1,\ldots,X_d\}$, let $\FppX$ denote the algebra of formal power series in $X$,
and let $I(X)$ be the augmentation ideal of $\FppX$, that is, the two-sided ideal $(X_1,\ldots,X_d)$.
Then there is an isomorphism of topological algebras 
\blue{\begin{equation}\label{eq:group algebra and power series iso}
\phi\colon \F_p\dbl S\dbr\to \FppX\qquad x_i\mapsto 1+X_i.
\end{equation}}
The composition of the embedding $S\hookrightarrow\F_p\dbl S\dbr$ with the isomorphism $\phi$ is the {\sl Magnus
embedding} $\psi\colon S\hookrightarrow\FppX$, and one has 
\[
 S_{(n)}=\{x\in S\mid \psi(x)\in I(X)^n\}=\{x\in S\mid \phi(x-1)\in I(X)^n\}.
\]
for $n\geq1$.

Moreover, one has a canonical isomorphism of graded algebras 
\[ \gr\FppX:=\bigoplus_{n\geq0}I(X)^n/I(X)^{n+1}\overset{\sim}{\longrightarrow} \FpX, \]
so that for a series $f\in\FppX$ such that $f\in I(X)^n\smallsetminus I(X)^{n+1}$, one may consider 
the class $f+I(X)^{n+1}$ as a homogeneous polynomial of $\FpX$ of degree $n$.
Likewise, one may consider the quotient $S_{(n)}/S_{(n+1)}$ as a subspace of the space of homogeneous polynomials
on $X$ of degree $n$
(cf., e.g., \cite{cq:onerel}*{\S~3.1}).

\begin{defin}\label{defin:initial form}\rm
For an element $x\in S_{(n)}\smallsetminus  S_{(n+1)}$, the class $\psi(x)+I(X)^{n+1}$, considered as a
homogeneous polynomial of $\FpX$ of degree $n$, is called the {\sl initial form} of $x$ in $\FpX$.
\end{defin}

In particular, the initial form of $x_i$ is $X_i$, and the initial form of the commutator $[x_i,x_j]=x_i^{-1}x_j^{-1}x_ix_j$
is the algebra commutator $[X_i,X_j]=X_iX_j-X_jX_i$, for every $i,j\in\{1,\ldots,d\}$.

\subsection{Restricted Lie algebras and graded group algebras}

A unitary associative $\F_p$-algebra $A$ may be turned into a Lie algebra by defining the Lie brackets
$[a,b]=ab-ba$ for all $a,b\in A$. 
The notion of restricted Lie algebra was introduced by N.~Jacobson (cf.~\cite{jac}*{\S~V.5}); we follow the definition given in {\cite{ddsms}*{\S~12.1}}.

\begin{defin}\label{defin:restricted}\rm
A Lie $\F_p$-algebra $\Lie$ is said to be a {\sl restricted Lie algebra} over $\F_p$ if there exists a unitary associative
$\F_p$-algebra $A$ and a monomorphism of Lie algebras $\vartheta\colon \Lie\to A$ (with $A$ considered as Lie algebra)
such that $\vartheta(v)^p\in\image(\vartheta)$ for every $v\in\Lie$. 
In particular, every unitary associative algebra can be considered as restricted Lie algebra.

The algebra $A$ together with the monomorphism $\vartheta$ is called a {\sl universal envelope} of $\Lie$ if it satisfies the 
following universal property: for every morphism of restricted Lie algebras $\phi\colon \Lie\to B$, where $B$ is a unitary
associative $\F_p$-algebra considered as a restricted Lie algebra, there exists a unique morphism of unitary associative $\F_p$-algebras
$\tilde\phi\colon A\to B$ such that $\phi=\tilde\phi\circ\vartheta$.
\end{defin}

Let $G$ be a finitely generated pro-$p$ group.
By \eqref{eq:prop zassenhaus}, the graded space $\Lie(G)$ endowed with the Lie brackets induced by commutators,
and $p$-power map induced by the $p$-power of elements, is a restricted Lie algebra.

Let $I_G$ denote the {\sl augmentation ideal} of the group $\F_p$-algebra $\F_p[G]$ associated to $G$.
Namely, $I_G$ is the two-sided ideal of $\F_p[G]$ generated by the elements $g-1$, $g\in G$.
The {\sl graded group $\F_p$-algebra} $\gr\F_p[G]$ associated to $G$ is the graded algebra
\[
 \gr\F_p[G]=\bigoplus_{n\geq0}I_G^n/I_G^{n+1},\qquad I_G^0=\F_p[G].
\]
One has the following result by A.~Jennings (see, e.g., \cite{ddsms}*{\S~12.2}).

\begin{prop}\label{prop:grFG}
 Let $G$ be a finitely generated pro-$p$ group.
Then for $g\in G_{(n)}\smallsetminus G_{(n+1)}$, the assignment 
\[
 g G_{(n+1)}\longmapsto (g-1)+I_G^{n+1}\in I_G^n/I_G^{(n+1)}
\]
induces a monomorphism of restricted Lie algebras $\Lie(G)\to\gr\F_p[G]$, so that $\gr\F_p[G]$ is the restricted envelope of $\Lie(G)$.
\end{prop}

Given $X=\{X_1,\ldots,X_d\}$, consider $\FpX$ as a restricted Lie algebra.
The Lie subalgebra $\Lie(X)$ of $\FpX$ generated by $X$ is a free restricted Lie algebra,
and $\FpX$ is its universal envelope.
Moreover, for $S$ a free pro-$p$ group with basis $\calX=\{x_1,\ldots,x_d\}$, we may identify $\Lie(S)$ with $\Lie(X)$ \blue{ via the isomorphism that carries the initial form of each $x_i$ to $X_i$},
and $\gr\F_p[S]$ with $\FpX$ \blue{ via the graded version of the isomorphism \eqref{eq:group algebra and power series iso}} (cf.~\cite{jochen:mild}*{Rem.~2.3}).

\begin{rem}\label{rem:LmodR}\rm
 Let $\rlie$ be a restricted ideal of the free restricted Lie algebra $\Lie(X)$ --- i.e., $\rlie$ is an
ideal of the Lie algebra $\Lie(X)$, considered as subalgebra of $\FpX$, with the further condition that $f^p\in\rlie$
for every $f\in\rlie$ ---, and set $\Lie=\Lie(X)/\rlie$.
Then the universal envelope $\calU_{\Lie}$ of $\Lie$ is isomorphic to the quotient $$\FpX/\mathcal{R}=\calU_{\Lie(X)}/\mathcal{R},$$
where $\mathcal{R}$ is the two-sided ideal of $\calU_{\Lie(X)}$ (identified with $\FpX$) generated by $\rlie$ 
(cf.~\cite{jochen:mild}*{Prop.~2.1}).
\end{rem}


\section{Mild pro-$p$ groups}
\label{sec:results}

A short exact sequence of pro-$p$ groups 
\begin{equation}\label{eq:presentation}
 \xymatrix{ \{1\}\ar[r] & R\ar[r] &S\ar[r] &G\ar[r] &\{1\}},
\end{equation}
with $S$ a free pro-$p$ group,
is called a {\sl presentation} of the pro-$p$ group $G$.
If $R\subseteq S_{(2)}$, then the presentation \eqref{eq:presentation} is {\sl minimal} --- roughly speaking,
$S$ and $G$ have the ``same'' minimal generating system.
In this case, a minimal set of elements of $S$ which generates $R$ as normal subgroup is called a set of {\sl defining
relations} of $G$.
A pro-$p$ group is {\sl finitely presented} if has a minimal presentation \eqref{eq:presentation} where $S$ is finitely generated
and the number of defining relations is finite.
 

\subsection{Cohomology of pro-$p$ groups}

For a pro-$p$ group $G$ we shall denote the {Galois} $\F_p$-cohomology groups simply by $H^n(G)$ for every $n\geq0$.
Thus $H^0(G)=\F_p$ and, given a minimal presentation \eqref{eq:presentation}, one has the following isomorphisms of vector spaces:
\begin{equation}\label{eq:H1H2}
 \begin{split}
  H^1(G) &\cong H^1(S)\cong(G/G_{(2)})^*, \\ H^2(G) &\cong H^1(R)^G\cong(R/R^p[R,S])^*    \end{split}  \end{equation}
(cf.~\cite{nsw:cohn}*{Prop.~3.9.1 and Prop.~3.5.9}).
In particular, $\dim(H^1(G))$ is the minimal number of generators of $G$, and $\dim(H^2(G))$ is the number of 
defining relations of $G$\blue{, which are therefore invariants of $G$}.
For a pro-$p$ group $G$, the cohomological dimension of $G$ is the number
\[
 \cd(G)=\max\{n\geq1\mid H^n(G)\neq0\}
\]
(cf.~\cite{nsw:cohn}*{Def.~3.3.1}).

The $\F_p$-cohomology of a pro-$p$ group comes endowed with the {\sl cup product} 
\[
 \xymatrix{ H^i(G)\times H^j(G)\ar[r]^-{\cup} & H^{i+j}(G) },
\]
which is graded-commutative, i.e.,
$$\alpha\cup\beta=(-1)^{ij}\beta\cup\alpha$$ for $\alpha\in H^i(G)$ and $\beta\in H^j(G)$.
For further facts on cohomology of pro-$p$ groups we refer to \cite{nsw:cohn}*{\S~III.9}.

\blue{T}he cup product from degree $1$ to degree $2$ in the $\F_p$-cohomology of a pro-$p$ group \blue{ is closely related to} the initial forms of the defining relations of the group itself (cf.~\cite{MPQT}*{Thm.~7.3 \blue{ and Cor.~7.4}}).

\begin{prop}\label{prop:H2}
Let $G$ be a finitely generated pro-$p$ group with minimal presentation \eqref{eq:presentation}.
The following are equivalent:
\begin{itemize}
 \item[(i)] the \blue{(bilinear extension of the) }cup product $H^1(G)^{\otimes2}\to H^2(G)$ \blue{ is surjective};
 \item[(ii)] one has the equality $R^p[R,S]=R\cap S_{(3)}$.
\end{itemize}
\blue{If (i) and (ii) hold}, given a\blue{ny} set of defining relations $\{r_1,\ldots,r_m\}\subseteq R$, for every $i=1,\ldots,m$ the initial form of $r_i$ is a homogeneous polynomial of degree $2$ (here we identify $\Lie(S)$ with $\Lie(X)$).
\end{prop}

\blue{In particular}, if a finitely generated pro-$p$ group has quadratic $\F_p$-cohomology,
then the two conditions in Proposition~\ref{prop:H2} hold.

\begin{rem}\label{rem:number relations}\rm
Let $G$ be a finitely generated pro-$p$ group with minimal presentation \eqref{eq:presentation}, and suppose that $G$ satisfies the two equivalent conditions in Proposition~\ref{prop:H2}.
Then 
\[
   \dim(R/R^p[R,S]) =\dim(RS_{(3)}/S_{(3)}) \leq\dim(S_{(2)}/S_{(3)})<\infty\]
namely, $G$ has a finite number of defining relations.
\end{rem}

Fix a basis $\calX=\{x_1,\ldots,x_d\}$ of $S$ and a set of defining relations $\{r_1,\ldots,r_m\}$ of $G$, and
set $X=\{X_1,\ldots,X_d\}$.
After identifying $\Lie(S)$ with $\Lie(X)$, we may identify the quotient $S/S_{(2)}$ with the space of homogeneous polynomial of
degree 1 in $\FpX$, and we may consider the quotient $RS_{(3)}/S_{(3)}$ as the subspace
of the space of homogeneous polynomial of $\FpX$ of degree 2 generated by the initial forms of the defining relations.
Then one has the following (cf.~\cite{MPQT}*{Thm.~F}).

\begin{prop}\label{prop:HG}
 Let $G$ be a finitely presented pro-$p$ group with minimal presentation \eqref{eq:presentation}.
If $H^\bullet(G)$ is quadratic, then it is isomorphic to the quadratic dual of the quadratic algebra
\[
 \bfQ_\bullet(S/S_{(2)},RS_{(3)}/S_{(3)})=\frac{\bfT_\bullet(S/S_{(2)})}{(RS_{(3)}/S_{(3)})}\cong\frac{\FpX}{\mathcal{R}},
\]
where $\mathcal{R}$ is the two-sided ideal generated by the initial forms of the defining relations of $G$.
\end{prop}


\subsection{Strongly free presentations}

The concept of mild group has been introduced by J.~Labute in \cite{labute:lie} and D.~Anick in \cite{anick} as a group
whose relators are as independent as possible with respect to the lower central series.
Mild pro-$p$ groups were subsequently defined for several filtrations
in \cite{labute:mild}, for $p$ odd, and in \cite{LM:mild2} and \cite{forre}, for $p=2$.
It came as a surprise when J.~Labute discovered that these groups have cohomological dimension $2$ and some occur naturally
as Galois groups of maximal $p$-extensions with restricted ramification.
Here we present the definition with respect to the $p$-Zassenhaus filtration.

Recall that one may consider the initial form of an element of a finitely generated free pro-$p$ group $S$ as a homogeneous
polynomial in $\Lie(X)\subseteq\FpX$.

\begin{defin}\label{defin:mild}\rm
\begin{itemize}
\item[(a)] Let $\rho_1,\dots,\rho_m$ be homogeneous elements of \blue{$\FpX$} of degrees $s_i=\deg \rho_i \blue{\ >1}$ and let $\mathcal{R}$ be the two-sided ideal of $\FpX$ generated by $\{\rho_1,\dots,\rho_m\}$.
The sequence $(\rho_1,\dots,\rho_m)$ is \emph{strongly free} if the Hilbert series of the quotient algebra
$A_\bullet=\FpX/\mathcal{R}$ is
\begin{equation}\label{eq:hilbert}
 h_{A_\bullet}(z)=\frac{1}{1-dz+(t^{s_1}+\dots+t^{s_m})}.
\end{equation}
The empty sequence is considered to be strongly free.
 \item[(b)]  A minimal finite presentation of a finitely generated pro-$p$ group \eqref{eq:presentation}, 
with defining relations $\{r_1,\ldots,r_m\}$,  
is called a \emph{strongly free presentation} (with respect to the $p$-Zassenhaus filtration)
if the sequence $(\bar r_1,\dots,\bar r_m)$ of the initial forms of the relations is strongly free.
A finitely presented pro-$p$ group that has a strongly free presentation is called a \emph{mild} group
(with respect to the $p$-Zassenhaus filtration).
\end{itemize}
\end{defin}

Mild pro-$p$ groups have exceptionally nice behavior both in terms of $\F_p$-cohomology and in terms of the graded group algebra,
as shown by the following (cf.~\cite{jochen:mild}*{Thm~2.12}). 

\begin{prop}\label{prop:mild}
 Let $G$ be a mild pro-$p$ group, with a strongly free presentation \eqref{eq:presentation} 
and defining relations $\{r_1,\ldots,r_m\}$.
Then
\begin{itemize}
 \item[(i)] $\cd(G)=2$;
 \item[(ii)] one has an isomorphism of graded $\F_p$-algebras $\gr\FppG\cong \FpX/\mathcal{R}$,
where $\mathcal{R}$ is the two-sided ideal generated by the initial forms of the defining relations
$\{r_1,\ldots,r_m\}$.
\end{itemize}
\end{prop}

\blue{As a consequence, the graded components of the $\F_p$-cohomology algebra $H^\bullet(G)$ of a mild pro-$p$ group $G$ are completely characterized by \eqref{eq:H1H2}.}


\subsection{Mild pro-$p$ groups and Koszul algebras}\label{ssec:mild koszul}

Now we can prove our main results.

\begin{thm}[\blue{Theorem~\ref{thm:main intro}}]\label{thm:mildkoszul}
Let $G$ be a mild pro-$p$ group, with strongly free presentation \eqref{eq:presentation}.
If $H^\bullet(G)$ is quadratic,
then the algebras $H^\bullet(G)$ and $\gr\FppG$ are quadratic dual to each other, and Koszul.
\end{thm}

\begin{proof}
\blue{Let $d=\dim(H^1(G))$ and $m=\dim(H^2(G))$. By Proposition \ref{prop:mild}, $\gr\FppG\cong\FpX/\mathcal{R}$. By definition of mildness, 
$$h_{\gr\FppG}(z)=h_{\FpX/\mathcal{R}}=\frac{1}{1-dz+mz^2},$$
hence, by Proposition \ref{prop:koszul A3}, $\gr\FppG$ is Koszul. By Proposition \ref{prop:HG}, $H^\bullet(G)\cong(\FpX/\mathcal{R})^!$, so $H^\bullet(G)$ is Koszul as well by Proposition \ref{prop:koszul}.}
\end{proof}

The converse of Theorem~\ref{thm:mildkoszul} gives a characterization of those mild pro-$p$ groups whose $\F_p$-cohomology algebra is quadratic.

\begin{prop}[\blue{Proposition~\ref{prop:mild cohom intro}}]\label{prop:mild cohom}
 Let $G$ be a finitely generated pro-$p$ group of cohomological dimension $2$ such that the $\F_p$-cohomology algebra $H^\bullet(G,\F_p)$ is Koszul.
 Then $G$ is a mild pro-$p$ group.
\end{prop}

\begin{proof}
Let \eqref{eq:presentation} be a minimal presentation of $G$, with a set of defining relations $\{r_1,\ldots,r_m\}$, with $m=\dim(R/R\cap S_{(3)})$.
Also, set $X=\{X_1,\ldots,X_d\}$, with $d=\dim(S/S_{(2)})$, and let $\rho_1,\ldots,\rho_m\in\FpX$ be the initial forms of the relations $r_1,\ldots,r_m$.

Since $H^\bullet(G)$ is Koszul, in particular it is a quadratic algebra, and by Proposition~\ref{prop:H2} the initial forms 
$\rho_1,\ldots,\rho_m\in\FpX$ are homogeneous polynomials of degree $2$.
Moreover, by Proposition~\ref{prop:HG}, one has an isomorphism of quadratic algebras
\begin{equation}\label{eq:proof mild cohom}
 A_\bullet:=\frac{\FpX}{(\rho_1,\ldots,\rho_m)}  \cong H^\bullet(G)^!,
\end{equation}
and by Proposition~\ref{prop:koszul} also the algebra $A_\bullet$ is Koszul.
By Proposition~\ref{prop:koszul A3} and by \eqref{eq:proof mild cohom}, one has
\[
 h_{A_\bullet}(z)=\frac{1}{h_{H^\bullet(G)}(-z)}=\frac{1}{1-dz+mz^2}.
\]
Namely, the sequence $(\rho_1,\ldots,\rho_m)$ is strongly free, and \eqref{eq:presentation} is a strongly free presentation.
Hence, $G$ is mild.
\end{proof}


\section{Examples}\label{sec:examples mild}

The results in \S~\ref{ssec:mild koszul} provide a big wealth of examples of pro-$p$ groups $G$ such that $H^\bullet(G)$ and $\gr\F_p[G]$ are Koszul, and quadratic dual to each other.

\begin{exam}[Demushkin groups]\label{ex:demushkin mild}\rm
Infinite Demushkin groups are mild (cf.~\cite{jochen:mild}*{Thm.~6.3}), and by definition they have quadratic
$\F_p$-cohomology  (see, e.g., \cite{nsw:cohn}*{Def.~3.9.9}).
In particular, if $G$ is a Demushkin group with $\dim(G/G_{(2)})=d$, then $H^\bullet(G)\cong \bfQ_\bullet(V,\Omega)$, where $V$ is a vector space of dimension $d$ with basis $\{a_1,\ldots,a_d\}$, and one of the following holds:
\begin{itemize}
 \item[(a)] $d$ is even, and $\Omega$ is generated by the quadratic relations
 \[
\left\{\begin{array}{c}
                a_1a_2=-a_2a_1=a_3a_4=-a_4a_3=\ldots=a_{d-1}a_d=-a_da_{d-1} \\
                a_ia_j=0\text{ for }(i,j)\neq(1,2),(2,1),(3,4),(4,3),\ldots,(d,d-1)
                   \end{array}  \right\} \]
                   (if $p\neq 2$ this is always the case);
\item[(b)] $d$ is even, $p=2$, and $\Omega$ is generated by the quadratic relations
\[
\left\{\begin{array}{c}
                a_1^2=a_1a_2=-a_2a_1=a_3a_4=-a_4a_3=\ldots=a_{d-1}a_d=-a_da_{d-1} \\
                a_ia_j=0\text{ for }(i,j)\neq(1,1),(1,2),(2,1),(3,4),(4,3),\ldots,(d,d-1)
                   \end{array}  \right\}; \]
\item[(b)] $d$ is odd, $p=2$, and $\Omega$ is generated by the quadratic relations
\[
\left\{\begin{array}{c}
                a_1^2=a_2a_3=-a_3a_2=a_4a_5=-a_5a_4=\ldots=a_{d-1}a_d=-a_da_{d-1} \\
                a_ia_j=0\text{ for }(i,j)\neq(1,1),(2,3),(3,4),(4,5),(5,4),\ldots,(d,d-1)
                   \end{array}  \right\}. \]
                   \end{itemize}
Hence, Theorem~\ref{thm:mildkoszul} gives another proof of the fact that the cohomology and the complete graded group algebra of an infinite Demushkin groups 
are quadratically dual to each other and Koszul (cf.~\cite{MPQT}*{\S~4.1}).
\end{exam}

The following example generalizes Demushkin groups.
\begin{exam}[One-relator pro-$p$ groups]\label{ex:1rel}\rm
\blue{Let} $S$ \blue{be} a finitely generated \blue{free} pro-$p$ group\blue{, with $p\neq2$,} \blue{let} $r\in S_{(2)}\smallsetminus S_{(3)}$, \blue{and let $R$ be the normal subgroup generated by $r$.}
Then the quotient $G=S/R$ is a mild pro-$p$ group
and it has quadratic $\F_p$-cohomology (cf.~\cite{cq:onerel}*{\S~4}.)
Hence, Theorem \ref{thm:mildkoszul}
gives another proof of the fact that the cohomology and the complete graded group algebra of such a one-relator pro-$p$ group 
are quadratic dual to each other and Koszul (cf.~\cite{cq:onerel}*{Prop.~4.6 and Prop.~4.9}).

\blue{\texttt{OPTIONAL }The combinatorial reason for the limitation $p\neq2$ is related to the fact that pro-$2$ groups may have first cohomology classes $\chi$ such that $\chi\cup\chi\neq0$; this may cause higher cohomology groups to be non-trivial, contradicting mildness (Prop.~\ref{prop:mild cohom}). At the extreme, the group $C_2=\langle x\mid x^2\rangle$, which is the only finite Demushkin group, has cohomology isomorphic to the polynomial $\F_2$-algebra in one variable.}
\end{exam}

%

\subsection{Testing for the quadraticity of cohomology}\label{ssec:quadraticity}
In general, thanks to the fact that any mild pro-$p$ group $G$ has cohomological dimension $2$, verifying the quadraticity of $H^\bullet(G)$ in Theorem \ref{thm:mildkoszul} amounts at checking only two conditions:
\begin{itemize}
\item[(a)] That each element of (a basis of) $H^2(G)$ is a linear combination of cup products of elements in $H^1(G)$. Since $\dim(H^2(G))$ equals the number of defining relations, this often reduces to checking that the initial forms of the defining relations are polynomials of degree $2$ and, as such, are linearly independent (cf.~e.g.~\cite{bglv:2rel}*{Prop.~1}). Dually, one can check that the defining relations induce cup products of elements of $H^1(G)$ which are linearly independent in $H^2(G)$.
\item[(b)] That, for every $\chi,\psi,\varphi\in H^1(G)$, the cup product $\chi\cup\psi\cup\varphi$, which is zero as $\cd (G)=2$, is zero \emph{purely as a consequence of what happens in $H^2(G)$}. In more formal terms, letting $\pi_\bullet: \bfT_\bullet(H^1(G))\to H^\bullet(G)$ be the canonical projection of graded algebras, $H^3(G)$ is included in the ideal $(\ker(\pi_2))$.
\end{itemize}

\begin{exam}\label{exam:mild quadratic}\rm
\blue{As an illustrative example, let $G$ be the pro-$p$ group with presentation
\[
 G=\left\langle\: x_1,\ldots,x_d\:\mid\: [x_1,x_2]=[x_1,x_3]=[x_2,x_3][x_4,x_5]=1 \:\right\rangle.
\]
In order to verify the mildness of $G$, we use {\sl Anick's criterion} \cite{anick}*{Thm.~3.2} (cf.~also \cite{jochen:mild}*{\S~3}).
The initial forms of the defining relations are
\begin{equation}\label{eq:initial forms}
 \rho_1=[X_1,X_2],\qquad \rho_2=[X_1,X_3],\qquad\rho_3=[X_2,X_3]+[X_4,X_5].
\end{equation}
We define the total order $X_1\prec X_2 \prec X_3\prec X_4\prec X_5$ and we extend it to the degree-lexicographic total order on the monomials in the indererminates $X=\{X_1,\ldots,X_5\}$, that is,
\[
\begin{array}{c}
X_{i_1}\cdots X_{i_a}\prec X_{j_1}\cdots X_{j_b}\\
\Updownarrow\\
a<b\mbox{ or }(a=b\mbox{ and }\exists k\in \{1,\dots,a\}(
X_{i_1}=X_{j_1},\dots,X_{i_{k-1}}=X_{j_{k-1}}, X_{i_k}\prec X_{j_k})).
\end{array} 
\]
The {\sl leading monomials} of the initial forms \eqref{eq:initial forms}, that is, the monomials with higher degree-lexicographic order,
are respectively $X_2X_1$, $X_3X_1$ and $X_5X_4$.
The sequence of the leading monomials is {\sl combinatorially free} (no leading monomial is a submonomial of a different one, and the beginning of a leading monomial is not the same as the ending of another, or the same, monomial, cf.~\cite{anick}*{\S~3}).
Hence, by Anick's criterion, the sequence $(\rho_1,\rho_2,\rho_3)$ is strongly free and therefore $G$ is mild. Proposition~\ref{prop:mild} gives $\gr\FppG\cong\FpX/\mathcal{R}$,
with $\mathcal{R}$ the two-sided ideal of $\FpX$ generated by $\rho_1,\rho_2,\rho_3$.

Let $\{\chi_1,\ldots,\chi_5\}$ be the basis of $H^1(G)$ dual to $\calX=\{x_1,\ldots,x_5\}$.
Then the set of the cup products dually associated to the defining relations is 
$\{chi_1\cup\chi_2,\chi_1\cup\chi_3,\chi_2\cup\chi_3\}$ (noticing that $\chi_4\cup\chi_5=\chi_2\cup\chi_3$ by the third relation). By applying the trace form of \cite{nsw:cohn}*{Prop.~3.9.13 (ii)} to a null linear combination of $\chi_1\cup\chi_2,\chi_1\cup\chi_3,\chi_2\cup\chi_3$, we see that the three cup products are linearly independent in $H^2(G)$. Since $\dim(H^2(G))=3$, this verifies Condition (a) above.
The usual graded-commutativity of the cup product gives
\[
\begin{matrix}
\chi_2\cup\chi_1=-\chi_1\cup\chi_2; & \chi_3\cup\chi_1=-\chi_1\cup\chi_3; \\
\chi_3\cup\chi_2=-\chi_2\cup\chi_3; & \chi_4\cup\chi_5=-\chi_5\cup\chi_4.
\end{matrix}
\]
\blue{All the remaining cup products} $\chi_i\cup\chi_j$ are zero, which automatically forces the triple cup products $\chi_i\cup\chi_j\cup\chi_k$ to be zero if $(i,j,k)$ is not a permutation of $(1,2,3)$. All the triple cup products such that $(i,j,k)$ is a permutation of $(1,2,3)$ are equal to $\pm(\chi_1\cup\chi_2\cup\chi_3)$ by graded-commutativity. Finally, one has
$$\blue{\chi_1\cup\chi_2\cup\chi_3=\chi_1\cup\chi_4\cup\chi_5=0\cup\chi_5=0.}$$
It is important to notice that the fact that all triple cup products are zero depends only on the linear structure of $H^2(G)$, that is, on the linear dependencies between double cup products. In other words, all triple cup products can be rewritten as linear combinations
\[
\sum_i k(\xi_i\cup\psi_i\cup\varphi_i),\qquad k\in\F_p,\quad\xi_i,\psi_i,\varphi_i\in H^1(G)
\]
such that, for all $i$, at least one among $\xi_i\cup\psi_i,\xi_i\cup\varphi_i,\psi_i\cup\varphi_i$ is already zero in $H^2(G)$.
This completes the verification of Condition (b).
Thus, $H^\bullet(G)$ is quadratic.

Theorem \ref{thm:mildkoszul} can therefore be applied, granting that $H^\bullet(G)$ and $\gr\FppG$ are Koszul and quadratic dual to each other.}
\end{exam}

With a similar argument, we can extend Example \ref{ex:1rel} to pro-$p$ groups with $2$ defining relations, provided $p$ is odd.
\begin{exam}[Mild 2-relator pro-$p$ groups]\label{ex:2rel}\rm
For $p\neq2$, let $S$ be a free pro-$p$ group with basis $\calX=\{x_1,\ldots,x_d\}$, and let
$r_1,r_2\in S_{(2)}\smallsetminus S_{(3)}$ be two elements such that their initial forms $\rho_1,\rho_2\in S_{(2)}/S_{(3)}$
are linearly independent.
Then by Proposition~\ref{prop:H2} and \cite{bglv:2rel}*{Thm.~1}, the quotient $G=S/R$, with $R$ the normal subgroup
generated by $r_1,r_2$, is a mild pro-$p$ group.

For $X=\{X_1,\ldots,X_d\}$ the set of indeterminates associated to $\calX$, one may choose $r_1,r_2$ such that
\[\begin{split}
   \rho_1 &= [X_1,X_2]+\sum_{\substack{i<j \\ (i,j)\succ(1,2)}}a_{ij}[X_i,X_j], \qquad \blue{a_{ij}\in\F_p,}\\
   \rho_2 &= [X_{s},X_t]+\sum_{\substack{i<j \\ (i,j)\succ(s,t)} }b_{ij}[X_i,X_j], \qquad \blue{b_{ij}\in\F_p,}  \end{split} \]
where $(s,t)\succ(1,2)$ and $a_{st}=0$ (here \blue{again} $\succ$ \blue{denotes} the \blue{degree-}lexicographic order on the set
of \blue{monomials in $X$}.
Let $\{\chi_1,\ldots,\chi_d\}$ be the basis of $H^1(G)$ dual to $\calX$.
Then by \cite{nsw:cohn}*{Prop.~3.9.13 (ii)} $\{\chi_1\cup\chi_2,\chi_s\cup\chi_t\}$ is a basis of $H^2(G)$.
\blue{A case by case analysis on triple cup products shows that Condition (b) above is satisfied. Hence $G$ has quadratic $\F_p$-cohomology, which is quadratic dual to $\gr\FppG\cong\FpX/(\rho_1,\rho_2)$, and $H^\bullet(G)$ and $\gr\FppG$ are koszul} (cf.~\cite{cq:2relUK}).
\end{exam}

\subsection{Groups and algebras from graphs}\label{ssec:graphs}
\blue{The reason why Condition (b) holds in Example \ref{ex:2rel} can be traced down to a shortage of pairs $(i,j)$ such that $\chi_i\cup\chi_j\neq 0$ in the second cohomology group. The idea can be captured systematically with the language of graphs.}

A finite simplicial graph $\Gamma=(\euV=\{v_1,\ldots,v_d\},\euE)$ is said to be {\sl triangle-free} if it has no ``triangles'' as induced subgraphs, i.e., \blue{$\euE$ does not include any triple
$ \{\{v_i,v_j\},\{v_j,v_k\},\{v_k,v_i\}\},$ for distinct $i,j,k\in\{1,\ldots,d\}.$}

\begin{exam}[Generalised right-angled Artin pro-$p$ groups]\label{ex:RAAG mild}\rm
A {\sl generalised right-angled Artin pro-$p$ group} {({\sl generalised pro-$p$ RAAG} for short)} associated to a graph $\Gamma=(\euV=\{v_1,\ldots,v_d\},\euE)$ is a pro-$p$ group $G$ with presentation
\[
 G=\left\langle\: v_1,\ldots,v_d\:\mid\: [v_i,v_j]=v_i^{\lambda_{ij}}v_j^{\mu_{ij}},\: i<j ,\:\{v_i,v_j\}\in\mathcal{E}\: \right\rangle,
\]
with $\lambda_{ij},\mu_{ij}\in p\Z_p$ for every $i,j$, and moreover $\lambda_{ij},\mu_{ij}\in 4\Z_2$ if $p=2$ (cf.~\cite{qsv:quad}*{\S~5.2}).
 If $\Gamma$ is triangle-free, then by \cite{qsv:quad}*{Thm.~F} $G$ is mild and one has \blue{an} isomorphisms of quadratic algebras
\[
 H^\bullet(G)\cong \bfLambda_\bullet(\Gamma),\qquad \gr\F_p[G]\cong \bfA_\bullet(\Gamma).
\]
\end{exam}

Some mild {generalised pro-$p$-RAAGs} with quadratic $\F_p$-cohomology are known not to occur as maximal pro-$p$ Galois groups of fields containing a root of 1 of order $p$.

\begin{exam}\label{exs:noGalois}\rm
Let $\Gamma=(\euV,\euE)$ be the simplicial graph with 
$$\euV=\{v_1,v_2,v_3,v_4\},\qquad \euE=\left\{\:\{v_1,v_2\},\{v_2,v_3\},\{v_3,v_4\},\{v_4,v_1\}\:\right\}$$
(i.e., $\Gamma$ is a square), and let $G$ be the associated { generalised pro-$p$-RAAG} with 
$\lambda_{ij}=\mu_{ij}=0$ for all $i,j$ --- namely, $G$ is isomorphic to the { direct} product of two $2$-generated free pro-$p$ groups.
By Example~\ref{ex:RAAG mild}, $G$ is mild, and in particular
\[
 \gr\F_p[G]\cong H^\bullet(G)^!\cong\frac{\FpX}{\left(\:[X_1,X_2],[X_2,X_3],[X_3,X_4],[X_4,X_1]\:\right)}. 
\]
Still, by \cite{cq:bk}*{Thm~5.6} $G$ does not occur as maximal pro-$p$ quotient of the absolute Galois
group of a field containing a root of unity of order $p$.
\end{exam}

\begin{exam}\label{exs:noGalois2}\rm
Let $\Gamma=(\euV,\euE)$ be the simplicial graph with 
$$\euV=\{v_1,v_2,v_3,v_4\},\qquad\euE=\left\{\:\{v_1,v_2\},\{v_2,v_3\},\{v_3,v_4\}\:\right\},$$
and let $G$ be the associated { generalised pro-$p$-RAAG} with 
$\lambda_{ij}=\mu_{ij}=0$ for all $i,j$.
By Example~\ref{ex:RAAG mild}, $G$ is mild, and in particular
\[
 \gr\F_p[G]\cong H^\bullet(G)^!\cong\frac{\FpX}{\left(\:[X_1,X_2],[X_2,X_3],[X_3,X_4]\:\right)}. 
\]
Still, by \cite{ip:RAAGs}*{Thm~1.2} $G$ does not occur as maximal pro-$p$ quotient of the absolute Galois
group of a field containing a root of unity of order $p$.
\end{exam}


\subsection{The cup-product criterion}\label{ssec:cupproductcrit}
\blue{There is a significant} cohomological criterion to check whether a finitely generated pro-$p$ group is mild \blue{independently of the choice of a group presentation (\cite{Schmidt}*{Thm.~5.5}, \cite{LM:mild2} and \cite{forre}, see also \cite{jochen:mild}*{\S~1}).}

\begin{prop}
Let 
$G$ be a finitely generated pro-$p$ group with $H^2(G)\neq0$ and such that \blue{$H^1(G)$ has a vector space decomposition} $H^1(G)=U\oplus W$  satisfying the following:
\begin{itemize}
 \item[(i)] The cup product induces a surjective map $\blue{U}\otimes W\to H^2(G)$; and 
 \item[(ii)] $\chi\cup\psi=0$ for every $\chi,\psi\in \blue{U}$. 
\end{itemize}
Then $G$ is mild.
\end{prop}

\blue{Inspired by the} above result\blue{, we establish }the following criterion for Koszulity, which holds in full generality for ``abstract'' quadratic algebras.

\begin{thm}
 Let $A_\bullet$ \blue{and $A^!_\bullet$ be a pair of quadratic dual }quadratic $\Bbbk$-algebra\blue{s}, and suppose that $A_1$ admits a decomposition as direct sum of vector spaces $A_1=U\oplus W$ such that:
 \begin{itemize}
  \item[(i)] $A_2=U\cdot W$;
  \item[(ii)] $U\cdot U=0$.
 \end{itemize}
Then $A_\bullet$ and $A^!_\bullet$ are Koszul and $\cd(A^!)=2$.
\end{thm}

\blue{We will show the Koszulity of $A_\bullet$ using the {\sl Rewriting Method} of \cite{lodval:book}*{\S~4.1}. For the interested reader, we mention in passing that this technique proves an even stronger property for $A_\bullet$, the PBW property,
which means that the ideals of $A_\bullet$ and $A^!_\bullet$ generated by the respective spaces of relators have quadratic Gr\"obner bases \cite{lodval:book}*{\S~4.3}.

\begin{proof}
Let $\{u_1,\dots,u_a\}$ be a basis of $U$ and $\{w_1,\dots,w_b\}$ be a basis of $W$. We order monomials according to the degree-lexicographic extension of the total ordering $v_1<v_2<\dots<v_a<w_1<w_2<\dots<w_b$. After some linear algebra manipulations (of the type leading to a {\sl normalized form}, in the sense of \cite{lodval:book}*{\S~4.1}), the conditions (i) and (ii) can be translated into the following set of defining relations for $A$, with the leading monomials written on the left hand side:
\begin{enumerate}
\item[(a)] $w_iw_j=\sum_t \lambda_tu_{a_t}w_{b_t}$ for all $i,j=1,\dots,b$, and suitable $\lambda_t\in \F$ (by (i));
\item[(b)] $w_iu_j=\sum_t \mu_tu_{r_t}w_{s_t}$ for all $i=1,\dots,b$ and $j=1,\dots,a$, and suitable $\mu_t\in \F$ (by (i));
\item[(c)] $u_iu_j=0$ for all $i,j=1,\dots,a$ (by (ii));
\item[(d)] possibly additional relations $u_iw_j=\sum_t \eta_tu_{l_t}w_{h_t}$ for some $i=1,\dots,a$ and $j=1,\dots,b$, and suitable $\eta_t\in \F$.
\end{enumerate}
The resulting critical monomials are
\begin{enumerate}
\item $w_iw_jw_k$, $\mbox{ all }i,j,k=1,\dots,b$;
\item $w_iw_ju_k$,  $\mbox{ all }i,j=1,\dots,b,\mbox{ all }k=1,\dots,a$;
\item $w_iu_ju_k$, $\mbox{ all }i=1,\dots,b, \mbox{ all }j,k=1,\dots,a$;
\item (possibly) $w_iu_jw_k$, $\mbox{ all }i=1,\dots,b,\mbox{ some }j=1,\dots,a,\mbox{ some }k=1,\dots,b$;
\item $u_iu_ju_k$, $\mbox{ all }i,j,k=1,\dots,a$;
\item (possibly) $u_iu_jw_k$, $\mbox{ all }i=1,\dots,a,\mbox{ some }j=1,\dots,a,\mbox{ some }k=1,\dots,b$;
\item (possibly) $u_iw_jw_k$, $\mbox{ some }i=1,\dots,a,\mbox{ some }j=1,\dots,b,\mbox{ all }k=1,\dots,b$;
\item (possibly) $u_iw_ju_k$, $\mbox{ some }i=1,\dots,a,\mbox{ some }j=1,\dots,b,\mbox{ all }k=1,\dots,a$.
\end{enumerate}
The corresponding rewriting graphs are as follows. We take some innocent freedom in naming indices in order to simplify the notation, as the only thing that matters is the ordered sequence of elements form $U$ and from $W$, not the particular elements themselves.
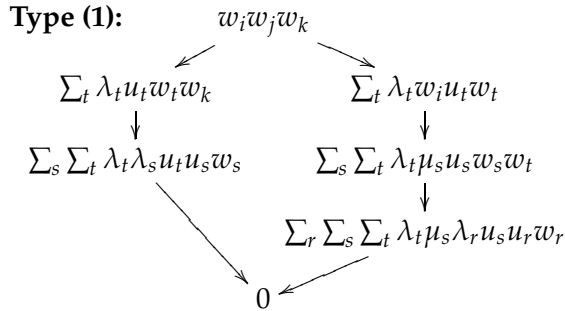
\begin{description}
\item[Type (1)] \hskip-4em$\xymatrix@R-14pt@C-40pt{&w_iw_jw_k\ar[dl]\ar[dr]&\\
\sum_t\lambda_tu_tw_tw_k\ar[d]&&\sum_t\lambda_tw_iu_tw_t\ar[d]\\
\sum_s\sum_t\lambda_t\lambda_su_tu_sw_s\ar[ddr]&&\sum_s\sum_t\lambda_t\mu_su_sw_sw_t\ar[d]\\
&&\sum_r\sum_s\sum_t\lambda_t\mu_s\lambda_ru_su_rw_r\ar[dl]\\
&0&}$
\medskip

\item[Type (2)] \hskip-4em$
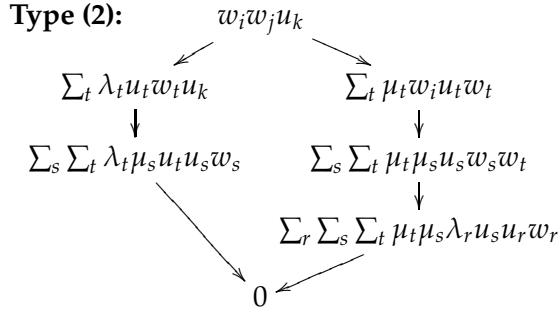
\xymatrix@R-14pt@C-40pt{&w_iw_ju_k\ar[dl]\ar[dr]&\\
\sum_t\lambda_tu_tw_tu_k\ar[d]&&\sum_t\mu_tw_iu_tw_t\ar[d]\\
\sum_s\sum_t\lambda_t\mu_su_tu_sw_s\ar[ddr]&&\sum_s\sum_t\mu_t\mu_su_sw_sw_t\ar[d]\\
&&\sum_r\sum_s\sum_t\mu_t\mu_s\lambda_ru_su_rw_r\ar[dl]\\
&0&}$
\end{description}
\begin{multicols}{2}
\begin{description}
\item[Type (3)] \hskip-4em$
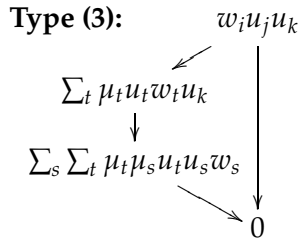
\xymatrix@R-14pt@C-40pt{&w_iu_ju_k\ar[dl]\ar[ddd]\\
\sum_t\mu_tu_tw_tu_k\ar[d] & \\
\sum_s\sum_t\mu_t\mu_su_tu_sw_s\ar[dr] & \\
&0}$
\item[Type (4)] \hskip-4em$
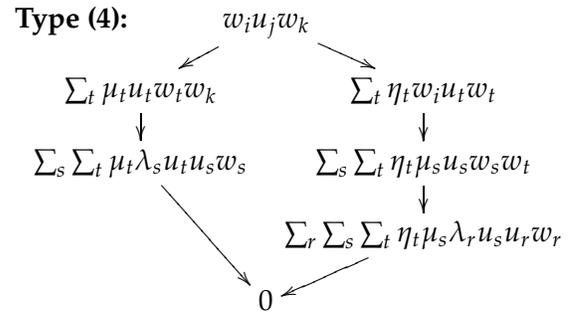
\xymatrix@R-14pt@C-40pt{&w_iu_jw_k\ar[dl]\ar[dr]&\\
\sum_t\mu_tu_tw_tw_k\ar[d]&&\sum_t\eta_tw_iu_tw_t\ar[d]\\
\sum_s\sum_t\mu_t\lambda_su_tu_sw_s\ar[ddr]&&\sum_s\sum_t\eta_t\mu_su_sw_sw_t\ar[d]\\
&&\sum_r\sum_s\sum_t\eta_t\mu_s\lambda_ru_su_rw_r\ar[dl]\\
&0&}$
\end{description}
\end{multicols}
\begin{multicols}{2}
\begin{description}
\item[Type (5)] $
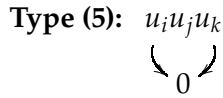
\xymatrix@R-14pt@C-30pt{u_iu_ju_k\ar@/_1pc/[d]\ar@/^1pc/[d]\\
0}$
\item[Type (6)]  $
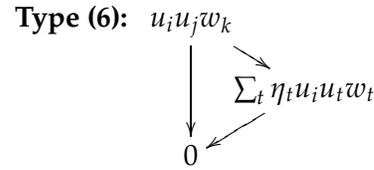
\xymatrix@R-14pt@C-30pt{u_iu_jw_k\ar[dd]\ar[dr]&\\
&\sum_t\eta_tu_iu_tw_t\ar[dl]\\
0}$
\end{description}
\end{multicols}
\begin{multicols}{2}
\begin{description}
\item[Type (7)] \hskip-4em$
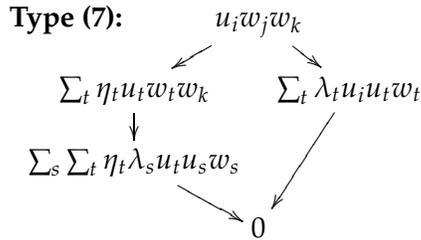
\xymatrix@R-14pt@C-40pt{&u_iw_jw_k\ar[dl]\ar[dr]&\\
\sum_t\eta_tu_tw_tw_k\ar[d] && \sum_t\lambda_tu_iu_tw_t\ar[ddl]\\
\sum_s\sum_t\eta_t\lambda_su_tu_sw_s\ar[dr] && \\
&0&}$
\item[Type (8)]  \hskip-4em$
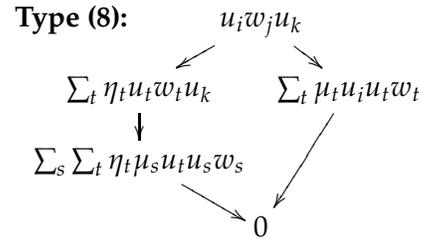
\xymatrix@R-14pt@C-40pt{&u_iw_ju_k\ar[dl]\ar[dr]&\\
\sum_t\eta_tu_tw_tu_k\ar[d] && \sum_t\mu_tu_iu_tw_t\ar[ddl]\\
\sum_s\sum_t\eta_t\mu_su_tu_sw_s\ar[dr] && \\
&0&}$
\end{description}
\end{multicols}
Since all the critical monomials are confluent, $A_\bullet$ is Koszul, as well as its quadratic dual.

By the same principle, thanks to condition (i) any degree $3$ monomial in $A_\bullet$ can be reduced to a linear combination of monomials containing two consecutive occurrences of elements in $U$, at which stage Condition (ii) forces them to be zero. Therefore $A_3=0$ and so, by Proposition \ref{prop:koszul}, $\cd(A^!_\bullet)=2$.
\end{proof}}


\subsection{Koszul algebras of elementary type}

The class of {\sl Koszul algebras of $H$-elementary type} is the class of those Koszul algebras which are constructible starting 
from trivial quadratic algebras and the cohomology algebras of Demushkin groups (cf.~Example~\ref{ex:demushkin mild}), and taking wedge products with the 1-generated trivial quadratic algebra $\bfQ_\bullet(\F_p,\F_p^{\otimes2})$ and direct sums.
Dually, the class of {\sl Koszul algebras of $G$-elementary type} is the class of those Koszul algebras which are constructible starting 
from free algebras and Demushkin algebras (cf.~Example~\ref{ex:demushkin algebra}), and taking symmetric tensor products with the 1-generated free algebra $\F_p[X_1]$ and free products (cf.~\cite{cq:onerel}*{Def.~5.8}).

\begin{rem}\label{remETC}\rm
The definition of Koszul algebras of $H$- and $G$-elementary type mimics the definition of {\sl pro-$p$ groups of elementary type}, given by I.~Efrat in \cite{ido:small}*{\S~3}: these groups are the finitely generated pro-$p$ groups which are constructible starting from $\Z_p$ and Demushkin groups, and taking certain semidirect products with $\Z_p$ and free products.
It is conjectured that if a finitely generated pro-$p$ group occurs as maximal pro-$p$ Galois group of a field containing a root of 1 of order $p$, then it is a pro-$p$ group of elementary type (cf.~\cite{ido:ETC}, see also \cite{ido:Q}*{Ques.~4.8} and \cite{marshall}*{\S~10}).
\end{rem}

For a quadratic algebra $A_\bullet$, set 
\[
 k(A_\bullet)=\max\{n\in\dbN\mid A_n\neq0\}.
\]
If $A_\bullet$ and $B_\bullet$ are quadratic algebras, then one has 
\begin{equation}
\begin{split}
 &k(A_\bullet\oplus B_\bullet)=\max\{k(A_\bullet),k(B_\bullet)\}, \\
 &k\left(A_\bullet\wedge \bfQ_\bullet(\F_p,\F_p^{\otimes2})\right)=k(A_\bullet)+1.
\end{split}
\end{equation}
Therefore, if $A_\bullet$ is a Koszul algebra of $H$-elementary type such that $k(A_\bullet)=2$, then either
\begin{itemize}
 \item[(a)] $A_\bullet\cong H^\bullet(G)$ for an infinite Demushkin group $G$;
 \item[(b)] or $A_\bullet\cong \bfQ_\bullet(V,V^{\otimes2})\wedge \bfQ_\bullet(\F_p,\F_p^{\otimes2})$, for some vector space $V$, $\dim(V)\geq1$;
 \item[(c)] or $A_\bullet\cong B_\bullet\oplus C_\bullet$, with $B_\bullet,C_\bullet$ Koszul algebras of $H$-elementary type such that $k(B_\bullet)=2$ and $k(C_\bullet)\leq2$.
\end{itemize}
Dually, if $A_\bullet$ is a Koszul algebra of $G$-elementary type such that $k(A^!_\bullet)=2$, then either
\begin{itemize}
 \item[(a')] $A_\bullet$ is an infinite Demushkin algebra;
 \item[(b')] or $A_\bullet\cong \FpX\otimes \F_p[X_0]$, with $X=\{X_1,\ldots,X_d\}$, $d\geq1$;
 \item[(c')] or $A_\bullet\cong B_\bullet\sqcup C_\bullet$, with $B_\bullet,C_\bullet$ Koszul algebras of $G$-elementary type such that $k(B_\bullet^!)=2$ and $k(C_\bullet^!)\leq2$;
\end{itemize}
respectively.

If the algebra $H^\bullet(G)$, respectively $\gr\F_p[G]$, associated to a finitely generated pro-$p$ group $G$, ends up in the above subclass of Koszul algebras of $H$-elementary type, respectively in the above subclass of Koszul algebras of $G$-elementary type, then $G$ is mild.

\begin{prop}
 Let $G$ be a finitely presented pro-$p$ group.
 The following are equivalent:
 \begin{itemize}
  \item[(i)] the $\F_p$-cohomology algebra $H^\bullet(G)$ is a Koszul algebra of $H$-elementary type, and $H^3(G)=0$;
\item[(ii)] the graded group algebra $\gr\F_p[G]$ is a Koszul algebra of $G$-elementary type, and $(\gr\F_p[G]^!)_3=0$.
 \end{itemize}
Moreover, if the above conditions hold, then $G$ is mild.
\end{prop}

\begin{proof}
 If condition~(i) holds, then $G$ is mild by Proposition~\ref{prop:mild cohom}.
 Then by Proposition~\ref{prop:mild} one has an isomorphism of quadratic algebras $\gr\F_p[G]\cong H^\bullet(G)^!$, and thus $\gr\F_p[G]$ is a Koszul algebra of $G$-elementary type.
 In particular, $(\gr\F_p[G]^!)_3\cong H^3(G)=0$.
 
Now suppose condition~(ii) holds, and set $\gr\F_p[G]=:A_\bullet=\bfQ_\bullet(A_1,\Omega)$, with $\Omega$ a subspace of $A_1^{\otimes2}$.
Since $A_\bullet$ is Koszul, also $A_\bullet^!$ is Koszul. 
Since $A_3^!=0$, one has 
\[\begin{split}
    h_{A_\bullet^!}(z) &= 1+\dim(A_1^!)z+\dim(A_2^!)z^2+0 \\ 
    &=1+\dim(A_1)z+\left(\dim(T_2(A_1)-\dim(A_2)\right)z^2 \\&=1+dz+mz^2
  \end{split}
\]
where $d=\dim(G/G_{(2)})=\dim(H^1(G))$ and $m=\dim(H^2(G))=\dim(\Omega)$.
Hence, Proposition~\ref{prop:koszul A3} implies 
\[
 h_{A_\bullet}(z)=\frac{1}{h_{A_\bullet^!}(-z)}=\frac{1}{1-dz+mz^2},
\]
i.e., the relations which generate $\Omega$ give rise to a strongly free sequence, and $G$ is mild.
Hence $H^3(G)=0$ by Proposition~\ref{prop:mild}, and $H^\bullet(G)\cong A_\bullet^!$. In particular, $H^\bullet(G)$ is a Koszul algebra of $H$-elementary type.
\end{proof}

In \cite{cq:onerel} it is conjectured that if $\K$ is a field containing a root of 1 of order $p$ (and containing also $\sqrt{-1}$ if $p=2$), then $H^\bullet(G)$ is a Koszul algebra of $H$-elementary type, and $\gr\F_p[G]$ is a Koszul algebra of $G$-elementary type (cf.~\cite{cq:onerel}*{Conj.~5.10}).
This conjecture is weaker than Efrat's conjecture (cf.~Remark~\ref{remETC}), but it is stronger than Conjectures~\ref{conj:posi}--\ref{ques:koszul}: namely, a positive answer to Efrat's conjecture would imply a positive answer to \cite{cq:onerel}*{Conj.~5.10}, which in turn would imply a positive answer to Conjectures~\ref{conj:posi}--\ref{ques:koszul}.

In particular, if the algebra $H^\bullet(G_{\K}(p))$ is a Koszul algebra of $H$-elementary type for every finitely generated maximal pro-$p$ Galois group $G_{\K}(p)$ with $\cd(G_{\K}(p))=2$, then $G_{\K}(p)$ is mild by Proposition~\ref{prop:mild cohom} and the Norm Residue Theorem.
This raises the following question.

\begin{question}
Let $\K$ be a field containing a root of 1 of order $p$ (and also $\sqrt{-1}$ if $p=2$), and suppose that $G_{\K}(p)$ is a finitely generated pro-$p$ group with $\mathrm{cd}(G_{\K}(p))=2$.
Is $G_{\K}(p)$ mild?
\end{question}


\end{document}